\documentclass[12pt]{amsart}
\usepackage{amssymb,amsmath,amsfonts,amsthm,nomencl,mathrsfs} 
\usepackage[arrow, matrix, curve]{xy}
\usepackage[latin1]{inputenc}
\usepackage{a4wide}
\renewcommand{\c}{\mathrm{c}}
\newcommand{\ILL}{\mathscr{L}}

\newcommand{\IL}{\mathsf{L}}

\renewcommand{\dim}{\mathrm{dim}}
\newcommand{\reg}{\mathrm{reg}}

\newcommand{\sing}{\mathrm{sing}}
\newcommand{\depth}{\mathrm{depth}}
\newcommand{\supp}{\mathrm{supp}}
\newcommand{\cod}{\mathrm{cod}}
\newcommand{\ver}{\mathrm{v}}
\renewcommand{\det}{\mathrm{det}}

\newcommand{\loc}{\mathrm{loc}}
\newcommand{\Id}{{\rm d}}

\newcommand{\IR}{\mathbb{R}}
\newcommand{\IT}{T}

\newcommand{\IN}{\mathbb{N}}

\newcommand*{\longhookrightarrow}{\ensuremath{\lhook\joinrel\relbar\joinrel\rightarrow}}

\newcommand{\f}{\frac}
\newcommand{\nn}{\nonumber}

\theoremstyle{plain}            % body italics
\newtheorem{theorem}{theorem}[section]

\newtheorem{Corollary}[theorem]{Corollary}
\newtheorem{Theorem}[theorem]{Theorem}
\newtheorem{Proposition}[theorem]{Proposition}

\theoremstyle{definition}       % body roman
\newtheorem{Definition}[theorem]{Definition}

\newtheorem{Remark}[theorem]{Remark}

\begin{document}
\begin{titlepage}
\title[$q$-parabolicity  of stratified pseudomanifolds and other singular spaces ]{$q$-parabolicity  of stratified pseudomanifolds and other singular spaces}

\author[F. Bei]{Francesco Bei}
\address{Franceso Bei, Institut f\"ur Mathematik, Humboldt-Universit\"at zu Berlin, 12489 Berlin, Germany} \email{bei@math.hu-berlin.de}

\author[B. G\"uneysu]{Batu G\"uneysu}
\address{Batu G\"uneysu, Institut f\"ur Mathematik, Humboldt-Universit\"at zu Berlin, 12489 Berlin, Germany} \email{gueneysu@math.hu-berlin.de}

\end{titlepage}

\maketitle

\begin{abstract}
The main result of this paper is a sufficient condition in order to have a compact Thom-Mather stratified pseudomanifold endowed with a $\hat{c}$-iterated edge metric on its regular part $q$-parabolic. Moreover, besides stratified pseudomanifolds, the $q$-parabolicity of other classes of singular spaces, such as compact complex Hermitian spaces, is investigated.
\end{abstract}

\section{Introduction}
Given $q\in [1,\infty)$, a smooth Riemannian manifold $(M,g)$ is called $q$-parabolic if the $q$-capacity
\begin{align*}
\mathrm{Cap}_{q,g}(K):=\inf\left \{ \int_M|\Id f|_{g^*}^q\Id\mu_g:\ f\in \mathsf{Lip}_{\c}(M),\ f\geq 1\ \text{on}\ K\right \}
\end{align*}
of each compact $K\subset M$ vanishes, $\mathrm{Cap}_{q,g}(K)=0$. The latter property is easily seen to be equivalent to the existence of a sequence of cut-off functions $\{\psi_n\}\subset  \mathsf{Lip}_{\c}(M)$, such that $0\leq \psi_n\leq 1$ for all $n$, $\left\|\Id \psi_n\right\|_{\IL^q\Omega^1(M,g)}\to 0$ as $n\to\infty$, and
\begin{align}
\nonumber \text{ for each compact $K\subset M$ there exists $n_K\in\IN$ such that $\psi_n\mid_K=1$ for all $n\geq n_K$}.
\end{align}
The importance of this concept stems at least from two reasons: Firstly,  the $2$-parabolicity of $(M,g)$ is equivalent to $g$-Brownian motion being recurrent \cite{grio} (in particular nonexplosive). Secondly, see \cite{troyan},  given a number $1<q<\infty$ and a continuous compactly supported $ 0\not\equiv h:M\to\IR$, the nonlinear $q$-Laplace equation
$$
 \Id^{\dagger_g} \left(|\Id u|_{g^*}^{q-2}\Id u\right)=h 
$$
has a weak solution in the space of $u\in\mathsf{W}^{1,q}_{\loc}(M)$ with $\left\|\Id u\right\|_{\IL^q\Omega^1(M,g)}<\infty$, if and only if $(M,g)$ is \emph{not} $q$-parabolic.\\
The aim of this paper is to investigate the $q$-parabolicity of some classes of \lq{}singular spaces\rq{} with particular regard to the case of compact smoothly Thom-Mather stratified pseudomanifolds. This is an important class of singular spaces whose study, from an analytic point of view, has been initiated by Cheeger in his seminal papers \cite{JCh}, \cite{JEC} and \cite{JC}.\\
In this setting we prove that a compact stratified pseudomanifold $X$ of dimension $m$ whose regular part (which is a smooth manifold) is equipped with an iterated edge metric of type $\hat{c}=(c_2,\dots,c_m)$ is $q$-parabolic (for some $q\in [1,\infty)$), if each singular stratum $Y$ of $X$ satisfies a certain compatibility criterion which only depends on $\hat{c}$, $m$, $q$ and $\dim(Y)$. In the most important case $\hat{c}=(1,\dots,1)$ (which e.g. covers many singular quotients of the form $M/G$ with $M$ a compact manifold and  $G$ a compact Lie group acting isometrically), this  result entails that these spaces are automatically $2$-parabolic and thus  stochastically complete. The importance of this class of metrics, as we will explain more precisely later, lies in its deep connection with the topology of $X$.\\ Besides to stratified pseudomanifolds, in this paper we investigate  also the $q$-parabolicity of other classes of singular  spaces such as compact Hermitian complex spaces, real algebraic varieties and almost complex manifolds endowed with a compatible and degenerate metric whose degeneration locus is a union of closed submanifolds  with codimension $\geq 2$. Let us point out that all of the above examples provide smooth Riemannian manifolds which are \emph{geodesically incomplete}, so that one cannot use Grigoryan\rq{}s well-known parabolicity and stochastic completeness criteria (cf. Theorem 11.8 and Theorem 11.14 in \cite{buch}) which require geodesic completeness and volume control. Our approach is more in the spirit of \cite{masa}.\\
This paper is organized as follows: In the second section we recall the main definitions and some important properties concerning parabolicity of Riemannian manifolds. The third  section, which contains the main result of this paper, deals with compact stratified pseudomanifolds.  The forth section is divided in three parts. The first one contains some technical statements that will be extensively used through the rest of the paper. The second part  deals with almost complex manifolds endowed with a compatible and degenerate metric whose degeneration locus is a union of closed submanifolds  with codimension $\geq 2$. Finally the third part  tackles the case of compact Hermitian complex spaces and real algebraic varieties .\\

{\bf Acknowledgements:} The authors wish to thank the anonymous referee for his valuable comments that in particular led us to the current formulation of Prop. \ref{Open}. The second named author (B.G.) would like to thank Stefano Pigola for many motivating discussions. Both autors have been financially supported by SFB 647: Raum-Zeit-Materie.

\section{Background}

Let $(M,g)$ be a smooth Riemannian manifold and let $\mu_g$ be the Riemannian volume measure. Let  $g^*$ be the smooth metric on $\IT^* M$ induced by $g$ which is given locally by $(g^*_{ij}):=(g_{ij})^{-1}$.  Let us label by $\mathsf{Lip}(M,g)$ the space of Lipschitz functions on $(M,g)$ and let $\mathsf{Lip}_{\c}(M)$ be the space of Lipschitz functions with compact support.  For $q\in [1,\infty]$ let  $\mathsf{W}^{1,q}(M,g)$ be the Sobolev space of functions which are in $\IL^q(M,g)$ and whose distributional differential lies in $\mathsf{L}^q\Omega^{1}(M,h)$. For $1\leq q< \infty$ let $\mathsf{W}^{1,q}_0(M,g)$ be the closure of $\mathsf{C}^{\infty}_{\c}(M)$ in $\mathsf{W}^{1,q}(M,g)$.

\begin{Definition}
Let $D\subset M$ be a relatively compact subset. Then the \emph{$q$-capacity of $D$ w.r.t. $g$}, $1\leq q<\infty$, is defined as 
\begin{equation}
\mathrm{Cap}_{q,g}(D):=\inf\left \{ \int_M|\Id f|_{g^*}^q\Id\mu_g:\ f\in \mathsf{Lip}_{\c}(M),\ f\geq 1\ \text{on}\ D\right \}.
\end{equation}
\end{Definition}

The following equivalence is well known, see for instance \cite{ALA} pag. 47 for the case $q=2$ or \cite{Troya} Prop. 4.1. 

\begin{Proposition}
\label{equivalence}
Let $(M,g)$ be a smooth Riemannian manifold. The following two properties are equivalent:
\begin{itemize}
\item For every compact subset $D\subset M$ we have $$\mathrm{Cap}_{q,g}(D)=0.$$
\item There exists a sequence of functions $\{\phi_n\}\subset \mathsf{C}^1_{\c}(M)$ such that  $0\leq \phi_n\leq 1$, $\phi_n\rightarrow 1$ uniformly on every compact subset as $n\rightarrow \infty$, and $\int_M|\Id\phi|_{g^*}^q\Id\mu_g\rightarrow 0$ as $n\rightarrow \infty$.
\end{itemize}
\end{Proposition}

\begin{Definition}
A smooth Riemannian manifold $(M,g)$ is called $q$-parabolic if it satisfies the equivalent conditions of Prop. \ref{equivalence}; we recall further that $2$-parabolic Riemannian manifolds are sometimes simply called \emph{parabolic}.
\end{Definition}

The next proposition  shows  that the characterization given in Prop. \ref{equivalence} can be largely relaxed. A proof can be found in \cite{troyan} Th. 3. Here we provide a different proof (which, as the results from \cite{gold}, has the advantage of being appliable to a much more general and possibly nonlocal setting).
\begin{Proposition}
\label{improved}
Let $(M,g)$ be a smooth Riemannian manifold. Then $(M,g)$ is $q$-parabolic if and only if there exists a sequence of functions $\{\psi_n\}\subset \mathsf{W}_{0}^{1,q}(M,g)$ such that 
\begin{itemize}
\item $0\leq \psi_n\leq 1$,
\item $\psi_n\rightarrow 1$ $\mu_g$-a.e. as  $n\rightarrow \infty$,
\item $\int_M|\Id\psi|_{g^*}^q\Id\mu_g\rightarrow 0$ as $n\rightarrow \infty$.
\end{itemize}
\end{Proposition}

\begin{proof}
Clearly we only have to prove that if there exists the asserted sequence of cut-off functions then $(M,g)$ is $q$-parabolic. 
%To this end, for any relatively compact $K\subset X$, we define the $\Gamma-\IL^q$ capacity\footnote{See in particular also \cite{hu} where for $q=2$ an analogous $0$-capacity has been defined} of $K$ as follows  
%\begin{align}\label{dfvv}
%\mathrm{cap}_{\Gamma,q}(K):=\inf\big\{\left\|\Gamma(f)^{1/2}\right\|_{\rho,q}\left|\> f\in \IAA, \text{$f\geq 1$ on $K$}\big\}\right.\in [0,\infty),
%\end{align}
%which is a well-defined and finite quantity in view of the regularity of $\Gamma$.
%As a first step  we are going to prove that  implies 
%\begin{align}\label{frtjz}
%\mathrm{cap}_{q}(K)=0\>\text{ for every open relatively compact $K\subset X$}.
%\end{align} 
%which in turn implies the asserted existence of cut-off functions (to see the latter assertion, by the topological assumptions on $X$, we can take an open relatively compact exhaustion $X=\bigcup_{l\in\IN} K_l$. As $\mathrm{cap}_{\Gamma, q}(K_l)=0$ for all $l$, it follows that for all $l,n\in\IN$ there is a $\phi_{l,n}\in \IAA$ such that $\phi_{l,n}\geq 1$ in $K_l$, $\IEE_{\Gamma,\mu,q}(\phi_{l,n})<1/n$. Then, in view of the Dirichlet property of $\Gamma$, $\phi_n:=(0\vee\phi_{n,n})\wedge 1 $ does the job).\\
%Let us now prove  \eqref{frtjz}.
To this end, we first extend the capacity to arbitrary Borel sets $Y\subset M$ as follows,
\begin{align}
\mathrm{Cap}_{q,g}(Y)=\sup\big\{\mathrm{Cap}_{q,g}(K)\left|\>\text{ $K\subset Y$, $K$ is relatively compact in $M$}\big\}\right.\in [0,\infty].\label{gfffz}
\end{align}
Then $Y\mapsto \mathrm{Cap}_{q,g}(Y)$ has the following three properties:
\begin{align}\nn
&\bullet\>\>Y_1\subset Y_2, \>\text{$Y_j$ Borel} \Longrightarrow\mathrm{Cap}_{q,g}(Y_1)\leq \mathrm{Cap}_{q,g}(Y_2),\\
\nn &\bullet\>\>Y_n\subset M\text{ open, relatively compact for all $n\in\IN$}\\\nn &\>\>\>\>\>\>\Longrightarrow 
	\mathrm{Cap}_{q,g}\left(\bigcup_{n\in\IN}Y_n\right)\leq \sum_{n\in\IN} 	\mathrm{Cap}_{q,g}\left(Y_n\right),\\\nn
&\bullet\>\>\mathrm{Cap}_{q,g} \{|f|>a\}\leq \f{2 }{a } \left\|\Id f\right\|_{\IL^q\Omega^1(M,g)}\>\text{ for any $f\in  \mathsf{Lip}_{\c}(M)$, $a>0$}. 
	\end{align}
The first property is trivial, the second one follows from (\ref{gfffz}) and the following simple inequality 
$$
\mathrm{Cap}_{ q,g}\left(\bigcup_{n\leq m}Y_n\right)\leq \sum_{n\leq m} 	\mathrm{Cap}_{ q,g}\left(Y_n\right)\text{ for all $m<\infty$,}
$$
and the last property follows noting that  $\min\{\frac{2f}{a}, 1\}$ is a test function for the relatively compact (open) set $\{|f|>a\}$.\\
Consider  now the sequence $\{\psi_n\}\subset \mathsf{W}_{0}^{1,q}(M,g)$. %be a sequence  with (i), (ii), (iii) as in Definition \ref{kfk} a). 
As for any $\psi_n$ there is a sequence $\phi_{l,n}\in \mathsf{Lip}_{\c}(M)$ with $\left\|\phi_{n,l}-\psi_n\right\|_{\mathsf{W}^{1,q}_0(M,g)}\to 0$ as $l\to\infty$, we can find a sequence $\phi_{n}\in \mathsf{Lip}_{\c}(M)$ with \begin{align}\label{sobb}
\left\|\phi_{n}-\psi_n\right\|_{\mathsf{W}^{1,q}_0(M,g)}\leq 1/n\>\text{ for all $n$}. 
\end{align}
Now we fix an arbitrary open relatively compact subset $K\subset M$. From (\ref{sobb}) and the second property in the statement we get
\begin{align}\nn
\left\|\Id\phi_n\right\|_{\IL^q\Omega^1(M,g)}\to 0,\>\>\left\|1_K(\phi_{n}-1\right)\|_{\IL^q(M,g)}\to 0.
\end{align}
In particular we can take subsequence $\phi_n\rq{}$ of $\phi_n$ (which depends on $K$) and a Borel set $Y_K\subset K$ with $\mu(Y_K)=0$ such that $\phi_n\rq{}(x)\to 1$ for all $x\in K\setminus Y_K$.  Of course $\phi_n\rq{}$ still satisfies $\left\|\Id \phi_n\rq{}\right\|_{\IL^q\Omega^1(M,g)}\to 0 $. Thus, given an arbitrary $\epsilon>0$, we can pick a subsequence $\tilde{\phi_n}$ of $\phi_n\rq{}$ (which depends on $K$ and $\epsilon$), such that 
\begin{align*}
\left\|\Id\tilde{\phi_n}\right\|_{\IL^q\Omega^1(M,g)}\leq \epsilon/n^2\text{ for all $n$ and}\>\>
\tilde{\phi_n}(x)\to 1\> \text{for all $x\in K\setminus Y_K$.}
\end{align*}
The convergence  $\tilde{\phi_n}(x)\to 1$ for all $x\in K\setminus Y_K$ implies 
$$
K\setminus Y_K\subset \bigcup_{ n\in\IN}\{\tilde{\phi_n}>1/2\},
$$
so that using the above properties of the capacity we get
$$
\mathrm{Cap}_{q,g}(K)=\mathrm{Cap}_{q,g}(K\setminus Y_K)\leq \mathrm{Cap}_{q,g}\left(\bigcup_{n\in\IN}\{\tilde{\phi_n}>1/2\}\right)\leq \sum_{n=1}^{\infty}\mathrm{Cap}_{q,g}(\{\tilde{\phi_n}>1/2\})\leq 4  \epsilon  ,
$$
where we have used $\mu(Y_K)=0$ and that $K$ is open. %and $\mathrm{supp}(\mu)=X$ for the first equality.
Thus, taking $\epsilon\to 0$ we arrive at $\mathrm{Cap}_{q,g}(K)=0$. So far we have shown that if $K\subset M$ is open and relatively compact then $\mathrm{Cap}_{q,g}(K)=0$. Now consider   an open relatively compact exhaustion $M=\bigcup_{l\in\IN} K_l$. As $\mathrm{Cap}_{q,g}(K_l)=0$ for all $l$, it follows that for all $l,n\in\IN$ there is a $\phi_{l,n}\in \mathsf{Lip}_{\c}(M)$ such that $\phi_{l,n}\geq 1$ in $K_l$, $\|\Id\phi_{l,n}\|_{\IL^q\Omega^1(M,g)}<1/n$. Then, using the sequence  $\phi_n:=\min\{1,\max\{0,\phi_{n,n}\}\}\in \mathsf{Lip}_{\c}(M)$, we can conclude that $\mathrm{Cap}_{q,g}(D)=0$ for every compact subset of $M$.
\end{proof}

Now we recall the following definition:

\begin{Definition}
Let $(M,g)$ be a smooth Riemannian manifold, let $H_g\geq 0$ be the Friedrichs extension of the Laplace-Beltrami operator $-\Delta_g|_{\mathsf{C}^{\infty}_{\c}(M)}:=(\Id^{\dagger_g}\Id)|_{\mathsf{C}^{\infty}_{\c}(M)}$ in $\IL^2(M,g)$, and let 
$$
(\mathrm{e}^{-tH_g})_{t\geq 0}\subset \ILL(\IL^2(M,g))
$$
be the corresponding heat-semigroup, defined a priori by the spectral calculus. Then $(M,g)$ is said to be \emph{stochastically complete}, if one has\footnote{$\mathrm{e}^{-tH_g}$ has a smooth integral kernel which satifies $\int_M \mathrm{e}^{-tH_g}(x,y)\Id\mu_g(y)\leq 1$ for all $t>0$, $x\in M$, so that one can define $\mathrm{e}^{-tH_g}f(x)$ for bounded functions $f$ on $M$ by $\mathrm{e}^{-tH_g}f(x):=\int_M \mathrm{e}^{-tH_g}(x,y)f(y)\Id\mu_g.$} $\mathrm{e}^{-tH_g}1(x)=1$ for all $t>0$, $x\in M$.
\end{Definition}

The name \lq\lq{}stochastic completeness\rq\rq{} stems from the classical fact that this property is equivalent to Brownian motions on $(M,g)$ being nonexplosive. We close this section recalling the following  properties:

\begin{Proposition}
Let $(M,g)$ be a smooth Riemannian manifold. If $(M,g)$ is parabolic then it is stochastically complete. If $\mu_g(M)<\infty$ then $(M,g)$ is parabolic if and only if it is stochastically complete.
\end{Proposition}

\begin{proof} The first property is a classical fact  (cf. \cite{masa}). The second property has been noted in a more general context in \cite{lenz}. We give a simple proof within our Riemannian framework: If $\mu_g(M)<\infty$ then $1=\mathrm{e}^{-tH_g}1\in \mathrm{Dom}(H_g)\subset \mathsf{W}^{1,2}_0(M,g)$. Therefore there exists a sequence $\psi_n\in \mathsf{C}_{\c}^{\infty}(M)$ such that $\psi_n\rightarrow 1$ in $\mathsf{W}^{1,2}_0(M,g)$. Defining $\phi_n:=\min\{1,\psi_n\}\in \mathsf{Lip}_{\c}(M)$ we get a sequence that makes $(M,g)$ parabolic. 
\end{proof}

\section{Stratified pseudomanifolds with iterated edge metrics}

In this section we come to the main result of this paper. We start by  briefly recalling the definition of smoothly stratified pseudomanifold with a Thom-Mather stratification. First we recall that, given a topological space $Z$, $C(Z)$ stands for the cone over $Z$ that is $Z\times [0,2)/\sim$ where $(p,t)\sim (q,r)$ if and only if $r=t=0$.

\begin{Definition}   
\label{thom}
 A smoothly Thom-Mather-stratified pseudomanifold $X$  of dimension $m$  is a metrizable, locally compact, second countable space which admits a locally finite decomposition into a union of locally closed strata $\mathfrak{G}=\{Y_{\alpha}\}$, where each $Y_{\alpha}$ is a smooth, open and connected manifold, with dimension depending on the index $\alpha$. We assume the following:
\begin{enumerate}
\item[(i)] If $Y_{\alpha}$, $Y_{\beta} \in \mathfrak{G}$ and $Y_{\alpha} \cap \overline{Y}_{\beta} \neq \emptyset$ then $Y_{\alpha} \subset \overline{Y}_{\beta}$
\item[(ii)]  Each stratum $Y$ is endowed with a set of control data $T_{Y} , \pi_{Y}$ and $\rho_{Y}$ ; here $T_{Y}$ is a neighborhood of $Y$ in $X$ which retracts onto $Y$, $\pi_{Y} : T_{Y} \rightarrow Y$
is a fixed continuous retraction and $\rho_{Y}: T_{Y}\rightarrow [0, 2)$  is a continuous  function in this tubular neighborhood such that $\rho_{Y}^{-1}(0) = Y$ . Furthermore,
we require that if $Z \in \mathfrak{G}$ and $Z \cap T_{Y}\neq \emptyset$  then
$(\pi_{Y} , \rho_{Y} ) : T_{Y} \cap Z \rightarrow Y\times [0,2) $
is a proper smooth submersion.
\item[(iii)] If $W, Y,Z \in \mathfrak{G}$, and if $p \in T_{Y} \cap T_{Z} \cap W$ and $\pi_{Z}(p) \in T_{Y} \cap Z$ then
$\pi_{Y} (\pi_{Z}(p)) = \pi_{Y} (p)$ and $\rho_{Y} (\pi_{Z}(p)) = \rho_{Y} (p)$.
\item[(iv)] If $Y,Z \in \mathfrak{G}$, then
$Y \cap \overline{Z} \neq \emptyset \Leftrightarrow T_{Y} \cap Z \neq \emptyset$ ,
$T_{Y} \cap T_{Z} \neq \emptyset \Leftrightarrow Y\subset \overline{Z}, Y = Z\ or\ Z\subset \overline{Y} .$
\item[(v)]  For each $Y \in \mathfrak {G}$, the restriction $\pi_{Y} : T_{Y}\rightarrow Y$ is a locally trivial fibration with fibre the cone $C(L_{Y})$ over some other stratified space $L_{Y}$ (called the link over $Y$ ), with atlas $\mathcal{U}_{Y} = \{(\phi,\mathcal{U})\}$ where each $\phi$ is a trivialization
$\pi^{-1}_{Y} (U) \rightarrow U \times C(L_{Y} )$, and the transition functions are stratified isomorphisms  which preserve the rays of each conic
fibre as well as the radial variable $\rho_{Y}$ itself, hence are suspensions of isomorphisms of each link $L_{Y}$ which vary smoothly with the variable $y\in U$.
\item[(vi)] For each $j$ let $X_{j}$ be the union of all strata of dimension less or equal than $j$, then 
$$
X_{m-1}=X_{m-2}\ \text{and  $X\setminus X_{m-2}$ dense\ in $X$}
$$
\end{enumerate}
\end{Definition}

The \emph{depth} of a stratum $Y$ is largest integer $k$ such that there is a chain of strata $Y=Y_{k},...,Y_{0}$ such that $Y_{j}\subset \overline{Y_{j-1}}$ for $1\leq j\leq k.$ A stratum of maximal depth is always a closed subset of $X$.  The  maximal depth of any stratum in $X$ is called the \emph{depth of $X$} as stratified spaces.
 Consider the filtration
\begin{equation}
X = X_{m}\supset X_{m-1}= X_{m-2}\supset X_{m-3}\supset...\supset X_{0}.
\label{pippo}
\end{equation}
 We refer to the open subset $X\setminus X_{m-2}$ of a smoothly Thom-Mather-stratified pseudomanifold $X$ as its regular set, and the union of all other strata as the singular set,
$$\reg(X):=X\setminus \sing(X)\ \text{where}\ \sing(X):=\bigcup_{Y\in \mathfrak{G}, \depth(Y)>0 }Y. $$
Given two Thom-Mather smoothly stratified pseudomanifolds  $X$ and $X'$,  a stratified isomorphism between them is a
homeomorphism $F: X\rightarrow X'$ which carries the open strata of $X$ to the open strata of $X'$
diffeomorphically, and such that $\pi'_{F(Y )}\circ  F = F \circ \pi_Y$ , $\rho'_{F(Y)}\circ F=\rho_Y$  for all $Y\in \mathfrak{G}(X)$.
For more details, properties and comments we refer to \cite{ALMP}, \cite{BHS}, \cite{RrHS},  \cite{JMA} , \cite{ver}, and also \cite{pflaum}. Here we point out that a large class of topological space such as irreducible complex analytic spaces or  quotient of manifolds through a proper Lie group action belong to this class of spaces.\\ 
Now we proceed introducing the class of smooth Riemmanian metrics on $\reg(X)$ which we are interested in. The definition is given by induction on the depth of $X$.  We label by $\hat{c}:=(c_2,...,c_m)$ a $(m-1)$-tuple of non negative real numbers. In order to state this  definition we need to recall recall that, given two Riemannian metrics $g$ and $h$ on a manifold $M$, $g$ and $h$ are said to be \emph{quasi-isometric}, briefly $g\sim h$, if there exists a real number $c>0$ such that $c^{-1}h\leq g\leq ch$.

\begin{Definition}
\label{iter}
 Let $X$ be a smoothly Thom-Mather-stratified pseudomanifold and let $g$ be a Riemannian metric on $\reg(X)$. If $\depth(X)=0$, that is $X$ is a smooth manifold, a $\hat{c}$-iterated  edge metric is understood to be any smooth Riemannian metric on $X$. Suppose now that $\depth(X)=k$ and that the definition of $\hat{c}$-iterated  edge metric is given in the case $\depth(X)\leq k-1$; then we call a smooth Riemannian metric $g$ on $\reg(X)$ a  $\hat{c}$-\emph{iterated  edge metric}  if it satisfies the following properties:
\begin{itemize}
\item Let $Y$ be a stratum of $X$ such that $Y\subset X_{i}\setminus  X_{i-1}$; by   definition \ref{thom} for each $q\in Y$ there exist an open neighbourhood $U$ of $q$ in $Y$ such that 
$$
\phi:\pi_{Y}^{-1}(U)\longrightarrow U\times C(L_{Y})
$$
 is a stratified isomorphism; in particular, 
$$
\phi:\pi_{Y}^{-1}(U)\cap \reg(X)\longrightarrow U\times \reg(C(L_{Y}))
$$
is a smooth diffeomorphism. Then, for each $q\in Y$, there exists one of these trivializations $(\phi,U)$ such that $g$ restricted on $\pi_{Y}^{-1}(U)\cap \reg(X)$ satisfies the following properties:
\begin{equation} 
(\phi^{-1})^{*}(g|_{\pi_{Y}^{-1}(U)\cap \reg(X)})\sim  \Id r^2+h_{U}+r^{2c_{m-i}}g_{L_{Y}}
\label{yhn}
\end{equation}
 where $m$ is the dimension of $X$, $h_{U}$ is a Riemannian metric defined over $U$ and  $g_{L_{Y}}$ is a $(c_2,...,c_{m-i-1})$-\emph{iterated edge metric}  on $\reg(L_{Y})$, $\Id r^2+h_{U}+r^{2c_{m-i}}g_{L_{Y}}$ is a Riemannian metric of product type on $U\times \reg(C(L_{Y}))$ and with $\sim$ we mean \emph{quasi-isometric}. 
\end{itemize}
\label{edge}
\end{Definition} 
When $\mathrm{depth}(X)=1$ we will say that $g$ is a $\hat{c}$-\emph{edge metric} and when $\mathrm{depth}(X)=1$ and $\hat{c}=1$ we will say that $g$ is a \emph{edge metric}. We remark that in \eqref{yhn} the neighborhood  $U$ can be chosen sufficiently small so that it is diffeomorphic to $(0,1)^i$ and $h_U$ it is quasi-isometric to the Euclidean metric restricted on $(0,1)^i$. Moreover we point out that with this kind of Riemannian metrics we have $\mu_g(\reg(X))<\infty$ in case $X$ is compact. There is the following nontrivial existence result:

\begin{Proposition}
Let $X$ be a smoothly Thom-Mather-stratified pseudomanifold of dimension $m$. For any $(m-1)$-tuple  of positive numbers $\hat{c}=(c_2,...,c_m)$, there exists  a smooth Riemannian metric on $\reg(X)$ which is a $\hat{c}$-iterated edge metric.
\end{Proposition}
\begin{proof}
See \cite{BHS} or \cite{ALMP} in the case $\hat{c}=(1,...,1,...,1)$.
\end{proof}

The importance of this class of metrics lies on its deep connection with the topology of $X$. In fact, as pointed out by Cheeger in his seminal paper \cite{JCh} (see also  \cite{FBe}, \cite{FrB}, \cite{FraBei}, \cite{HMa} and \cite{MNa}  for further developments ) the $L^2$-cohomology of $\reg(X)$ associated to an iterated edge metric is isomorphic to the intersection cohomology  of $X$ associated with a perversity depending only on $\hat{c}$. In other words the $L^2$-cohomology of this kind of metrics (which a priori is an object that lives only  on $\reg(X)$) provides a non trivial topological information of the whole space $X$.

\begin{Theorem}
\label{straticom}
Let $X$ be a compact smoothly Thom-Mather-stratified pseudomanifold of dimension $m$. Let $q\in [1,\infty)$ and let $g$ be a smooth Riemmanian metric on $\reg(X)$ such that $g$ is a $\hat{c}$-iterated edge metric with $\hat{c}=(c_2,...,c_m)$. Assume that for every singular stratum $Y$ of $X$ one has
\begin{equation}
\label{condition1}
c_{m-i}\cdot(m-i-1)\geq q-1, \quad\text{where $i:=\dim(Y)$,}  
\end{equation}
and that moreover for every singular stratum $Y$ of $X$ with $\depth(Y)>1$, one has
\begin{equation}
\label{condition2}
c_{m-i}\cdot(m-i-1-q)>-1, \quad\text{where again $i:=\dim(Y)$.} 
\end{equation}
Then $(\reg(X), g)$ is $z$-parabolic for each $z\in [1,q]$. In particular, given $z\in [1,q]$ and a continuous compactly supported function $ 0\not\equiv h:\reg(X)\to\IR$, the nonlinear $z$-Laplace equation
$$
 \Id^{\dagger_g} \left(|\Id u|_{g^*}^{z-2}\Id u\right)=h 
$$
has no weak solution $u\in\mathsf{W}^{1,z}_{\loc}(\reg(X))$ with $\left\|\Id u\right\|_{\IL^q\Omega^1(\reg(X),g)}<\infty$.
\end{Theorem}

Before to give a proof of the above theorem we recall the following proposition.
\begin{Proposition}
\label{partition}
Let $X$ be a smoothly Thom-Mather-stratified pseudomanifold , and let $\mathcal{U}_{A}=\{U_{\alpha}\}_{\alpha\in A}$ be an open cover of $X$. Then  there is a bounded partition of unity with bounded differential subordinate to $\mathcal{U}_{A}$, meaning that there exists a family of functions $\lambda_{\alpha}:X\rightarrow [0,1], \alpha\in A$ such that
\begin{enumerate}
\item Each $\lambda_{\alpha}$ is continuous and $\lambda_{\alpha}|_{\reg(X)}$ is smooth.
\item  $\supp(\lambda_{\alpha})\subset U_{\alpha}$ for some $\alpha\in A$.
\item $\{\supp(\lambda_{\alpha})\}_{\alpha\in A}$ is a locally finite cover of $X$.
\item For each $x\in X$ one has $\sum_{\alpha\in A}\lambda_{\alpha}(x)=1$.
\item There are constants $C_{\alpha}<\infty$ such that each $\lambda_{\alpha}$ satisfies $\|\Id\lambda_{\alpha}|_{\reg(X)}\|_{\IL^{\infty}\Omega^1(\reg(X),g)}\leq C_{\alpha}$.
\label{ygn}
\end{enumerate}
\end{Proposition}
\begin{proof}
See for instance \cite{BYOU} Prop. 3.2.2.
\end{proof}

Now we are in position to prove Theorem \ref{straticom}.
\begin{proof}
First of all we remark that it is enough to show that $(\reg(X),g)$ is $q$-parabolic. For the remaining $z\in [1,q)$ the statement follows by the fact that $\mu_g(\reg(X))<\infty$. The proof is given by induction on $\depth(X)$. If $\depth(X)=0$ then $X$ is a smooth compact manifold and therefore the theorem holds. Assume now that $\depth(X)=b$ and that the theorem holds in the case $\depth(X)\leq b-1$. This step of the  proof is \emph{divided in two parts}: in the first  we construct a local model of our desired sequence. In the second part  we then patch together these local models in order to get a suitable sequence of Lipschitz  functions with compact support. 
Let $Y$ be a singular stratum of $X$ of dimension $i$ and let $L_Y$, $\pi_Y$ and $\rho_Y$ as in Def. \ref{thom}. Let $p\in Y$ and let $U_p$ be an open neighborhood of $p$ in $Y$ such that we have an isomorphism $\phi:\pi_Y^{-1}(U_p)\rightarrow U_p\times C(L_Y)$ which satisfies \eqref{yhn}.  In particular  we know that $(\phi^{-1})^{*}(g|_{\pi_{Y}^{-1}(U)\cap \reg(X)})\sim  \Id r^2+h_{U}+r^{2c_{m-i}}g_{L_{Y}}$ and that $g_{L_{Y}}$ is a $(c_2,...,c_{m-i-1})$- iterated edge metric on $\reg(L_{Y})$. Clearly $\depth(L_Y)\leq b-1$. We can reformulate \eqref{condition1} and \eqref{condition2} respectively in the following way 
\begin{equation}
\label{condition3}
c_{\cod(Y)}(\cod(Y)-1)\geq q-1\  \text{for every}\ Y\subset \sing(X)
\end{equation}
\begin{equation}
\label{condition 4}
c_{\cod(Y)}(\cod(Y)-1-q)>-1\  \text{for every}\ Y\subset \sing(X)\ \text{with}\ \depth(Y)>1.
\end{equation}
By the fact that  $\phi:\pi_Y^{-1}(U_p)\rightarrow U_p\times C(L_Y)$ is a stratified isomorphism  we have $\phi(U_p)=U_p\times \ver(C(L_Y))$, where $\ver(C(L_Y))$ is the vertex of $C(L_Y)$, $\phi(\reg(\pi_Y^{-1}(U_p)))=U_p\times \reg(C(L_Y))$ and finally if $Z$ is a singular stratum of $X$ such that $Z\cap \pi_{Y}^{-1}(U_p)\neq \emptyset$ then $\phi(Z\cap \pi_{Y}^{-1}(U_p))=U_p\times (0,2)\times W$ where $W$ is a singular stratum in $L_Y$. In particular $\depth(Z)=\depth(W)$ and $\cod(Z)=\cod(W)$. On the other hand, starting with a singular stratum $W'\subset L_Y$, we can find a singular stratum $Z'$ of $X$ such that $Z'\cap \pi_{Y}^{-1}(U_p)\neq \emptyset$,  $\phi(Z'\cap \pi_{Y}^{-1}(U_p))=U_p\times (0,2)\times W'$, $\depth(Z')=\depth(W')$ and $\cod(Z')=\cod(W').$
This implies that on $L_Y$,  with respect to the  $(c_2,...,c_{m-i-1})$-iterated edge metric  $g_{L_Y}$, we have 
\begin{equation}
\label{condition5} 
c_{\cod(W)}(\cod(W)-1)\geq q-1\ \text{for every }\ W\subset \sing(L_Y)
\end{equation}
\begin{equation}
\label{condition6}
c_{\cod(W)}(\cod(W)-1-q)>-1 \   \text{for every }\ W\subset \sing(L_Y)\ \text{with}\ \depth(W)>1.
\end{equation}
%Moreover we have  
%$c_{m-i-1-j}(m-i-j-2)\geq q-1$ for $j=0,...,m-i-3$. By the fact that $m-i-1=\dim(L_Y)$ we have  $c_{l_Y-j}(l_Y-j-1)\geq q-1$ for $j=0,...,l_Y-2$ where $l_Y$ stands for $\dim(L_Y)$.
We are  therefore in the position  to use  the inductive hypothesis and hence   we can conclude that $(\reg(L_Y),g_{L_Y})$ is $q$-parabolic.
 Let $\{\beta_{L_Y,n}\}$ be a sequence of compactly supported Lipschitz functions that makes $(\reg(L_Y),g_{L_Y})$ $q$-parabolic. If $\depth(Y)=1$ this means that $L_Y$ is  a  smooth compact manifold and $g_{L_Y}$ is a smooth Riemannian metric on $L_Y$. In this case we will always use the constant sequence $\{1\}$.
In order to pursue our aim we need now to define a suitable sequence of cut-off functions on $U_p\times C(L_Y)$. Let $\epsilon_n:=\frac{1}{n^2}$ and $\epsilon'_n:=\mathrm{e}^{-\frac{1}{\epsilon_n^2}}=\mathrm{e}^{-n^4}$. %Then on $\pi_Y^{-1}(U_p)$ we can define the following sequence of functions $\{\psi_{U_p,n}\}$ as 
% \begin{equation}
%\label{cutstrati}
%\psi_{U_p,n}:= \left\{
%\begin{array}{ll}
%1 & \rho_Y\geq \epsilon_n\ \text{on}\ \pi_Y^{-1}(U_p)\\
%(\frac{\rho_Y}{\epsilon_n})^{\epsilon_n} & 2\epsilon'_n\leq \rho_Y\leq \epsilon_n\ \text{on}\ \pi_Y^{-1}(U_p)\\
%(\frac{2\epsilon'_n}{\epsilon_n})^{\epsilon_n}(\frac{\rho_Y}{\epsilon'_n}-1) &  \epsilon'_n\leq \rho_Y\leq 2\epsilon'_n\ \text{on}\ \pi_Y^{-1}(U_p)\\
%0 & 0\leq \rho_Y\leq \epsilon'_n\ \text{on}\ \pi_Y^{-1}(U_p)
%\end{array}
%\right.
%\end{equation}
On $U_p\times C(L_Y)$ consider the following sequence of functions:
 \begin{equation}
\label{cutstratif}
\gamma_{U_p,n}:= \left\{
\begin{array}{llll}
1 & r\geq \epsilon_n\ \text{on}\ U_p\times C(L_Y)\\
(\frac{r}{\epsilon_n})^{\epsilon_n} & 2\epsilon'_n\leq r\leq \epsilon_n\ \text{on}\ U_p\times C(L_Y)\\
(\frac{2\epsilon'_n}{\epsilon_n})^{\epsilon_n}(\frac{r}{\epsilon'_n}-1) &  \epsilon'_n\leq r\leq 2\epsilon'_n\ \text{on}\ U_p\times C(L_Y)\\
0 & 0\leq r\leq \epsilon'_n\ \text{on}\ U_p\times C(L_Y)
\end{array}
\right.
\end{equation}
For $\Id \gamma_{U_p,n}|_{\reg(U\times C(L_Y))}$ we have the following estimate:
 \begin{equation}
\label{esdif}
|\Id \gamma_{U_p,n}|_{\reg(U\times C(L_Y))}|_{g^*}\leq  \left\{
\begin{array}{llll}
0 & r\geq \epsilon_n\ \text{on}\ U_p\times C(L_Y)\\
(\frac{r}{\epsilon_n})^{\epsilon_n-1} & 2\epsilon'_n\leq r\leq \epsilon_n\ \text{on}\ U_p\times C(L_Y)\\
(\frac{2\epsilon'_n}{\epsilon_n})^{\epsilon_n}(\frac{1}{\epsilon'_n}) &  \epsilon'_n\leq r\leq 2\epsilon'_n\ \text{on}\ U_p\times C(L_Y)\\
0 & 0\leq r\leq \epsilon'_n\ \text{on}\ U_p\times C(L_Y)
\end{array}
\right.
\end{equation}
where $|\bullet |_{g^*}$ in \eqref{esdif} is the pointwise norm that $\Id r^2+h_{U_p}+r^{2c_{m-i}}g_{L_Y}$ induces on $\mathrm{T}^*(\reg(U_p\times C(L_Y)))$. We want to show that 
\begin{equation}
\label{limiting}
\lim_{n\rightarrow \infty}\|\Id \gamma_{U_p,n}|_{\reg(U_p\times C(L_Y))}\|_{\IL^q\Omega^{1}(\reg(U_p\times C(L_Y)), dr^2+h_{U_p}+r^{2c_{m-i}}g_{L_Y})}=0.
\end{equation}
To this aim, using \eqref{esdif}, we have 
\begin{align}
\label{newver}
\|\Id \gamma_{U_p,n}|_{\reg(U_p\times C(L_Y))}\|^q_{\IL^q\Omega^{1}(\reg(U_p\times C(L_Y)), \Id r^2+h_{U_p}+r^{2c_{m-i}}g_{L_Y})}\leq\\\nn \int_{\epsilon'_{n}}^{2\epsilon'_{n}}\int_{U_p}\int_{\reg(L_Y)}\left(\frac{2\epsilon'_n}{\epsilon_n}\right)^{q\epsilon_n}\left(\frac{1}{\epsilon'_n}\right)^qr^{c_{m-i}(m-i-1)}\Id\mu_r\Id \mu_{h_{U_p}} \Id \mu_{g_{L_Y}}+\\\nn +\int_{2\epsilon'_{n}}^{\epsilon_{n}}\int_{U_p}\int_{\reg(L_Y)}\left(\frac{r}{\epsilon_n}\right)^{q\epsilon_n-q}r^{c_{m-i}(m-i-1)}\Id\mu_r\Id \mu_{h_{U_p}} \Id \mu_{g_{L_Y}}
\end{align}
For the first term on the right hand side of \eqref{newver} we have
\begin{align}
\label{onevolest}
&\int_{\epsilon'_{n}}^{2\epsilon'_{n}}\int_{U_p}\int_{\reg(L_Y)}\left(\frac{2\epsilon'_n}{\epsilon_n}\right)^{q\epsilon_n}\left(\frac{1}{\epsilon'_n}\right)^qr^{c_{m-i}(m-i-1)}\Id\mu_r\Id \mu_{h_{U_p}} \Id \mu_{g_{L_Y}}\\\nn
&=\mu_{h_{U_p}}(U_p)\mu_{g_{L_Y}}( \reg(L_Y))\int_{\epsilon'_{n}}^{2\epsilon'_{n}}\left(\frac{2\epsilon'_n}{\epsilon_n}\right)^{q\epsilon_n}\left(\frac{1}{\epsilon'_n}\right)^qr^{c_{m-i}(m-i-1)}\Id\mu_r\\\nn
&= \frac{\mu_{h_{U_p}}(U_p)\mu_{g_{L_Y}}( \reg(L_Y))}{c_{m-i}(m-i-1)+1}\left(\frac{2\epsilon'_n}{\epsilon_n}\right)^{q\epsilon_n}\left(\frac{1}{\epsilon'_n}\right)^q\left((2\epsilon'_n)^{c_{m-i}(m-i-1)+1}-(\epsilon'_n)^{c_{m-i}(m-i-1)+1}\right)\\\nn
&= \frac{\mu_{h_{U_p}}(U_p)\mu_{g_{L_Y}}( \reg(L_Y))}{c_{m-i}(m-i-1)+1}(2n^2\mathrm{e}^{-n^4})^{qn^{-2}}\mathrm{e}^{qn^4}\mathrm{e}^{-n^4(c_{m-i}(m-i-1)+1)}(2^{c_{m-i}(m-i-1)+1}-1)\\\nn
&=:\mu_{h_{U_p}}(U_p)\mu_{g_{L_Y}}(\reg(L_Y))a_{n,q}. 
\end{align}
It is straighforward to see that  $\lim_{n\rightarrow \infty}a_{n,q}=0$.
For the second term on the the right hand side of \eqref{newver} we have
\begin{align}
\label{secvolest}
& \int_{2\epsilon'_{n}}^{\epsilon_{n}}\int_{U_p}\int_{\reg(L_Y)}\left(\frac{r}{\epsilon_n}\right)^{q\epsilon_n-q}r^{c_{m-i}(m-i-1)}\Id\mu_r\Id \mu_{h_{U_p}} \Id \mu_{g_{L_Y}}\\\nn
&=(\frac{1}{\epsilon_n})^{q\epsilon_n-q}\mu_{h_{U_p}}(U_p)\mu_{g_{L_Y}}(\reg( L_Y))\int_{2\epsilon'_{n}}^{\epsilon_{n}}r^{q\epsilon_n-q+c_{m-i}(m-i-1)}\Id\mu_r\\\nn
&= \frac{\mu_{h_{U_p}}(U_p)\mu_{g_{L_Y}}(\reg( L_Y))}{q\epsilon_n-q+1+c_{m-i}(m-i-1)}\left(\frac{1}{\epsilon_n}\right)^{q\epsilon_n-q}(\epsilon_n^{q\epsilon_n-q+1+c_{m-i}(m-i-1)}-(2\epsilon'_n)^{q\epsilon_n-q+1+c_{m-i}(m-i-1)})\\\nn
&= \frac{\mu_{h_{U_p}}(U_p)\mu_{g_{L_Y}}( \reg(L_Y))}{qn^{-2}-q+1+c_{m-i}(m-i-1)}(n^2)^{qn^{-2}-q}\\\nn
&\>\>\>\>\>\>\times\left(\left(\f{1}{n^2}\right)^{qn^{-2}-q+1+c_{m-i}(m-i-1)}-(2\mathrm{e}^{-n^4})^{qn^{-2}-q+1+c_{m-i}(m-i-1)}\right)\\ \nn
&=:\mu_{h_{U_p}}(U_p)\mu_{g_{L_Y}}(\reg(L_Y))b_{n,q}.
\end{align}
Also in this case $\lim_{n\rightarrow \infty}b_{n,q}=0$. Hence we proved that \eqref{limiting} holds. Define now a sequence on $U_p\times C(L_Y)$ as 
\begin{equation}
\label{sulcono}
\alpha_{U_p,n}:=\gamma_{U_p,n}\beta_{L_{Y},n}.
\end{equation}
 We clearly have $\lim_{n\rightarrow \infty}\alpha_{U_p,n}(x)=1$ for every $x\in U_p\times C(L_Y)$. Over $U_p\times \reg(C(L_Y))$, for $\Id(\alpha_{U_p,n})$, we have  
$$
\Id\alpha_{U_p,n}=\gamma_{U_p,n}\Id\beta_{U_p,n}+\beta_{U_p,n}\Id\gamma_{U_p,n}
$$ and therefore 
\begin{align}
\nn & \|\Id \alpha_{U_p,n}\|_{\IL^q\Omega^{1}(\reg(U_p\times C(L_Y)), dr^2+h_{U_p}+r^{2c_{m-i}}g_{L_Y})}  \\
\nn & \leq \|\gamma_{U_p,n}\Id\beta_{U_p,n}\|_{\IL^q\Omega^{1}(\reg(U_p\times C(L_Y)), \Id r^2+h_{U_p}+r^{2c_{m-i}}g_{L_Y})}+\\ \nn & +\|\beta_{U_p,n}\Id\gamma_{U_p,n}\|_{\IL^q\Omega^{1}(\reg(U_p\times C(L_Y)), \Id r^2+h_{U_p}+r^{2c_{m-i}}g_{L_Y})}
\end{align}

According to \eqref{limiting} we have 
$$\lim_{n\rightarrow \infty}\|\beta_{U_p,n}\Id\gamma_{U_p,n}\|_{\IL^q\Omega^{1}(\reg(U_p\times C(L_Y)), \Id r^2+h_{U_p}+r^{2c_{m-i}}g_{L_Y})}=0.$$ 
For  $\gamma_{U_p,n}\Id\beta_{U_p,n}$ we argue in this way. If $\depth(Y)=1$ then $\beta_{U_p,n}=1$ for each $n\in \mathbb{N}$ and clearly
$$\lim_{n\rightarrow \infty}\|\gamma_{U_p,n}\Id\beta_{U_p,n}\|_{\IL^q\Omega^{1}(\reg(U_p\times C(L_Y)), \Id r^2+h_{U_p}+r^{2c_{m-i}}g_{L_Y})}=0.$$
If $\depth(Y)>1$ then we have 
 \begin{align}
\nn & \lim_{n\rightarrow \infty}\|\gamma_{U_p,n}\Id\beta_{U_p,n}\|^q_{\IL^q\Omega^{1}(\reg(U_p\times C(L_Y)), \Id r^2+h_{U_p}+r^{2c_{m-i}}g_{L_Y})} \leq \\
\nn & \lim_{n\rightarrow \infty}\|\Id\beta_{U_p,n}\|^q_{\IL^q\Omega^{1}(\reg(U_p\times C(L_Y)), \Id r^2+h_{U_p}+r^{2c_{m-i}}g_{L_Y})}=\\
\nn &\lim_{n\rightarrow \infty}\mu_h(U_p)^q\|\Id\beta_{U_p,n}\|^q_{\IL^q\Omega^1(\reg(L_Y),g_{L_Y})}\int_0^1r^{c_{m-i}(m-i-1-q)}\Id r=0
\end{align}
because $\int_0^1r^{c_{m-i}(m-i-1-q)}\Id r<\infty$. Summarizing we proved that 
\begin{equation}
\label{limit2}
\lim_{n\rightarrow \infty}\|\Id \alpha_{U_p,n}|_{\reg(U_p\times C(L_Y))}\|_{\IL^q\Omega^{1}(\reg(U_p\times C(L_Y)), \Id r^2+h_{U_p}+r^{c_{m-i}}g_{L_Y})}=0.
\end{equation}
Consider now the following   sequence $\{\psi_{U_p,n}\}$ on $\pi_Y^{-1}(U_p)$ defined  as 
\begin{equation}
\label{seqseq}
\psi_{U_p,n}:=\alpha_{U_p,n}\circ \phi^{-1}.
\end{equation}
We have again  $\lim_{n\rightarrow \infty}\alpha_{U_p,n}(x)=1$ for every $x\in \pi_Y^{-1}(U_p)$ and by \eqref{yhn} and \eqref{limit2}  we get 
\begin{equation}
\label{limit3}
\lim_{n\rightarrow \infty}\|\Id \psi_{U_p,n}|_{\pi^{-1}_Y(U_p)\cap \reg(X)}\|_{\IL^q\Omega^{1}(\pi^{-1}_Y(U_p)\cap \reg(X), g|_{\pi^{-1}_Y(U_p)\cap \reg(X)})}=0.
\end{equation}
This concludes the \emph{first part of the proof}. In fact over any open subset of $X$ satisfying \eqref{yhn} we constructed our desired sequence given by $\{\psi_{U_p,n}\}$. Now we begin \emph{the second part of the proof}. As previously explained, here the goal is gluing together the "local"  sequences $\{\psi_{U_p,n}\}$ in order to get a globally defined sequence of Lipschitz functions with compact support which makes $(\reg(X),g)$ $q$-parabolic. To this aim we need first to introduce a suitable partition of unity. Consider now the following closed subsets of $X$,
$$
K:=\overline{\bigcup_{Y\subset \sing(X)} T_Y\cap \rho_Y^{-1}([0,1))},\ \Omega:=X\setminus \left(\bigcup_{Y\subset \sing(X)} T_Y\cap \rho_Y^{-1}([0,1))\right).
$$
By the fact that $X$ is compact we can find a finite set of points $\mathfrak{T}:=\{p_1,...,p_s\}\subset \sing(X)$ such that the following properties are satisfied. For each $p_i$ there is an open neighborhood $U_{p_i}\subset Y_i$, the singular stratum containing $p_i$, such that \eqref{yhn} holds and such that $\{\pi_Y^{-1}(U_{p_i})\cap K, i=1,...,s\}$ is a finite  open cover of $K$.  
%Consider the  following two closed subsets of $X$ contained in  $\reg(X)$:
%$$\Omega_1:=X\setminus \bigcup_{i=1}^s(\pi_{Y_i}^{-1}(U_{p_i})\cap \rho_{Y_i}^{-1}[0,1))\ \text{and}\ \Omega_2:= X\setminus \bigcup_{i=1}^s\pi_{Y_i}^{-1}(U_{p_i}).$$ 
By construction  $\Omega$ is contained in $\reg(X)$. Let now $A\subset \reg(X)$ be an open subset such that $\Omega\subset A$.   In this way we get that 
$$
\mathfrak{M}:=\{\pi_{Y_1}^{-1}(U_{p_1}),...,\pi_{Y_s}^{-1}(U_{p_s}), A\}
$$
 is a finite open cover of $X$. According to Prop. \ref{partition} let $\mathfrak{L}:=\{\lambda_{\alpha}\}_{\alpha\in A}$ be a \emph{finite} partition of unity with bounded differential subordinated to $\mathfrak{M}$. Let us consider  the finite set of functions $\{\tau_1,..,\tau_s,\tau_A\}$ where $\tau_i$,  $i=1,...,s$, is defined as the sum of all functions $\lambda_{\alpha}\in \mathfrak{L}$ having support in $\pi_{Y_i}^{-1}(U_{p_i})$
and $\tau_A$ is defined as the sum of all functions $\lambda_{\alpha}\in \mathfrak{L}$ having support in $A$. Now, for each $\pi_{Y_i}^{-1}(U_{p_i})$, consider the sequence $\{\psi_{U_{p_i},n}\}$ as defined in \eqref{seqseq}. Finally define the sequence $\{\chi_n\}$ as 
\begin{equation}
\label{loureed}
\chi_n:=\tau_1\psi_{U_{p_1},n}+...+\tau_s\psi_{U_{p_s},n}+\tau_A.
\end{equation} 
We want to show that $\{\chi_n|_{\reg(X)}\}$ makes $(\reg(X),g)$ $q$-parabolic.
By construction $\chi_{n}|_{\reg(X)}$ is locally  Lipschitz. Let now $q\in \sing(X)$ and let $i\in\{1,...,s\}$. If $q\notin \supp(\tau_i)$ then $\tau_i\psi_{U_{p_i},n}$ is null on a neighborhood of $q$. If $q\in \supp(\tau_i)$ then $q\in\pi_{Y_i}^{-1}( U_{p_i})$ and using \eqref{yhn} we get $\phi(q)=(u,[r,y])$ with $u\in U_{p_i}$ and $[r,y]\in \sing(C(L_Y))$. We have $(\tau_i\psi_{U_{p_i},n})\circ \phi^{-1}=(\tau_i\circ \phi^{-1})\alpha_{U_{p_i},n}$ where $\alpha_{U_{p_i},n}$ is defined in \eqref{sulcono}. By construction $(\tau_i\circ \phi^{-1})\alpha_{U_{p_i},n}$ is null on a neighborhood (which depends on $n$) of $(u,[r,y])$ because $\tau_i\circ \phi^{-1}$ has compact support in $U\times C(L_Y)$, $\alpha_{U_{p_i},n}=\gamma_{U_{p_i},n}\beta_{U_{p_i},n}$, $\gamma_{U_{p_i},n}$ is null on a neighborhood of $\ver(C(L_Y))$ in $C(L_Y)$ and $\beta_{U_{p_i},n}$ is null on a neighborhood of $\sing(L_Y)$ in $L_Y$. Eventually this tells us that  $\chi_{n}$ is null on a neighborhood (which depends on $n$) of $\sing(X)$ because we have just shown that every single term on the right hand side of \eqref{loureed} is null on a neighborhood (which depends on $n$) of $\sing(X)$. Therefore each $\chi_{n}|_{\reg(X)}$ is Lipschitz with compact support.
Clearly we have $0\leq \chi_n\leq 1$ and $\lim_{n\rightarrow \infty}\chi_n|_{\reg(X)}=1$ pointwise. For $\|\Id\chi_n|_{\reg(X)}\|_{\IL^q\Omega^1(\reg(X),g)}$ we argue as follows: Over $\reg(X)$  we have 
\begin{equation}
\label{wall}
\Id\chi_n=\tau_1\Id\psi_{U_{p_1},n}+\psi_{U_{p_1},n}\Id\tau_1+...+\tau_s\Id\psi_{U_{p_s},n}+\psi_{U_{p_s},n}\Id\tau_s +\Id\tau_{A}.
\end{equation}
Therefore 
\begin{align}
\label{limitzero}
 \|\Id\chi_n|_{\reg(X)}\|_{\IL^q\Omega^1(\reg(X),g)}\leq  \|\tau_1\Id\psi_{U_{p_1},n}+...+\tau_s\Id\psi_{U_{p_s},n}\|_{\IL^q\Omega^1(\reg(X),g)}+\\\nn +\|\psi_{U_{p_1},n}\Id\tau_1+...+\psi_{U_{p_s},n}\Id\tau_s)+\Id\tau_A\|_{\IL^q\Omega^1(\reg(X),g)}
\end{align}
For the right hand side of \eqref{limitzero} we have 
\begin{align*}
&\|\tau_1\Id\psi_{U_{p_1},n}+...+\tau_s\Id\psi_{U_{p_s},n}\|_{\IL^q\Omega^1(\reg(X),g)}\\
&\leq  \|\tau_1\Id\psi_{U_{p_1},n}\|_{\IL^q\Omega^1(\reg(X),g)}+...+\|\tau_s\Id\psi_{U_{p_s},n}\|_{\IL^q\Omega^1(\reg(X),g)}.
\end{align*}

Using \eqref{limit3}  we get  for each $i=0,...,s$
\begin{equation}
\label{zerolimit}
\lim_{n\rightarrow \infty} \|\tau_1\Id\psi_{U_{p_i},n}\|_{\IL^q\Omega^1(\reg(X),g)}=0.
\end{equation}
For $$\psi_{U_{p_1},n}\Id\tau_1+...+\psi_{U_{p_s},n}\Id\tau_s+\Id\tau_A$$ we have 
\begin{align}
&\lim_{n\rightarrow \infty} \|\psi_{U_{p_1},n}\Id\tau_1+...+\psi_{U_{p_s},n}\Id\tau_s+\Id\tau_A\|_{\IL^q\Omega^1(\reg(X),g)}=\\\nn &\|\lim_{n\rightarrow \infty}(\psi_{U_{p_1},n}\Id\tau_1+...+\psi_{U_{p_s},n}\Id\tau_s+\Id\tau_A)\|_{\IL^q\Omega^1(\reg(X),g)}=\\\nn &\|\Id\tau_1+...+\Id\tau_s+\Id\tau_A\|_{\IL^q\Omega^1(\reg(X),g)}=\\\nn &\|\Id(\tau_1+...+\tau_s+\tau_A)\|_{\IL^q\Omega^1(\reg(X),g)}=\|\Id 1\|_{\IL^q\Omega^1(\reg(X),g)}=0.
\end{align}
In conclusion the sequence $\{\chi_n|_{\reg(X)}\}$  makes $(\reg(X),g)$ $q$-parabolic and so the proof of the theorem is completed.
\end{proof}

We close this section by adding some immediate consequences of Theorem \ref{straticom}. 

\begin{Remark}
 If  $(c_2,...,c_{m})=(1,...,1)$ then $(\reg(X),g)$ is $2$-parabolic and thus stochastically complete. In fact, then \eqref{condition1}, \eqref{condition2} becomes 
\begin{equation}
\label{condition4}
 \left\{
\begin{array}{ll}
\cod(Y)\geq 2 &\ \text{if}\ \depth(Y)=1\\
\cod(Y)>2 &\ \text{if}\ \depth(Y)>1
\end{array}
\right.
\end{equation} 
and, according to the Definition \eqref{thom}, \eqref{condition4} is clearly satisfied by every singular stratum $Y\subset \sing(X)$. These metrics have been considered in \cite{ALMP}.
\end{Remark}

A particular case of smoothly Thom-Mather-stratified pseudomanifolds is provided by \emph{manifolds with conical singularities}. A topological space $X$ is a manifold with conical singularities, if it is a metrizable, locally compact, Hausdorff space such that there exists a sequence of points $\{p_{1},...,p_{n},...\}\subset X$ which satisfies the following properties:
\begin{enumerate}
\item $X\setminus  \{p_{1},...,p_{n},...\}$ is a smooth open manifold.
\item For each $p_{i}$ there exists an open neighborhood $U_{p_i}$, a compact smooth manifold $L_{p_i}$ and a  map $\chi_{p_i}:U_{p_i}\rightarrow C_{2}(L_{p_i})$ such that $\chi_{p_i}(p_i)=v$ and 
$$ 
\chi_{p_i}|_{U_{p_i}\setminus  \{p_{i}\}}:U_{p_i}\setminus  \{p_i\}\longrightarrow L_{p_i}\times (0,2)
$$
is a smooth diffeomorphism. 
\end{enumerate}
Using the notations of Def. \ref{thom} this means that 
$$
X=X_n\supset X_{n-1}=X_{n-2}=...=X_1=X_0.
$$
In this case a $\hat{c}$-\emph{iterated edge metric} $g$ on $\reg(X)$ is a Riemannian metric on $\reg(X)$ with the following property: for each  conical point $p_i$ there exists a map $\chi_{p_i}$, as defined above,  such that 
\begin{equation}
\label{pianello}
(\chi_{p_i}^{-1})^*(g|_{U_{p_{i}}})\sim \Id r^2+r^{2c}h_{L{p_{i}}}
\end{equation}
 where $h_{L{p_{i}}}$ is a Riemannian metric on $L_{p_{i}}$ and $c>0$. When $c=1$, \eqref{pianello} is called conic metric while, when $c>1$, \eqref{pianello} is called horn metric.  Applying Theorem \ref{straticom} we get the following corollary.

\begin{Corollary}
Let $X$ be compact manifold with isolated conical singularities, and let $g$ be a smooth Riemannian metric on $\reg(X)$ which  satisfies \eqref{pianello}. Assume that $c(n-1)\geq q-1$. Then $(\reg(X),g)$ is $s$-parabolic for all $s\in[1,q]$. In particular, conic metrics and horn metrics are always $2$-parabolic.
\end{Corollary}

The next propositions provide other applications of Theorem \ref{straticom}.

\begin{Proposition}
Let $V\subset \mathbb{R}^m$ be an irreducible compact analytic surface with isolated singularities.  Let $g$ be the Riemannian metric on $\reg(V)$, the regular part of $V$, induced by the standard Euclidean metric on $\mathbb{R}^m$. Then $(\reg(V),g)$ is $q$-parabolic for all $q\in [1,2]$.
\end{Proposition}

\begin{proof}
The proposition follows combining Theorem 1.1 in \cite{DGR} with Theorem \ref{straticom}.
\end{proof}

Finally, we record a result concerning singular quotients. To this end, we recall that if $G$ is a compact Lie group acting isometrically on a smooth compact Riemannian manifold $(M,g)$, then $M/G$ canonically becomes a smoothly Thom-Mather-stratified pseudomanifold . Furthermore, with $\pi:M\rightarrow M/G$ the projection onto the orbit space, let $\pi_*g$ denote the smooth Riemannian metric on $\reg(M/G)$ which is induced by $g$ through $\pi$.

\begin{Proposition}
In the above situation, assume that the orbit space $M/G$ has no codimension one stratum. Then $(\reg(M/G),\pi_*g)$ is $q$-parabolic for all $q\in [1,2]$.
\end{Proposition}

\begin{proof}
If $M/G$ has no codimension one stratum then $\pi_*g$ is quasi-isometric to a $\hat{c}$-iterated edge metric with $\hat{c}=(1,...,1).$ This is showed in \cite{Sj}. Now the claim follows from applying Theorem \ref{straticom}. 
\end{proof}

\section{Other singular spaces}\label{riem}

\subsection{Further preliminary results}

\noindent This subsection concerns the stability of $q$-parabolicity. First we will recall that $q$-parabolicity is preserved under quasi-isometry. Then we will show that in the setting of Hermitian manifolds, in order to preserve $2$-parabolicity, it is enough a weaker condition than the quasi-isometry.  These results are of fundamental importance, as stochastic completeness itself is not preserved under quasi-isometry, see for instance \cite{TLY}. \vspace{2mm}

%In the Riemannian case we will simply say that \emph{$(X,g)$ is $q$-parabolic, if and only if $\Gamma_g$ is $q$-parabolic.} \vspace{2mm}

%\begin{Remark} 1. We recall that if $(X,g)$ is a smooth Riemannian manifold without boundary which is $\IL^2$-complete, then $\IEE_g$ is parabolic and thus stochastically complete.\\
%2. As explained in the introduction, if $X$ is connected, then given a number $1<q<\infty$ and a continuous compactly supported $h:X\to\IR$ with $h(x_0)\neq 0$ for some $x_0$, the nonlinear $q$-Laplace equation
%$$
%\Id^{\dagger} \left(|\Id f|^{q-2}\Id u\right)\mid_g=h 
%$$
%has a weak solution $u$ which satisfies  $u\in\mathsf{W}^{1,q}_{\loc}(M)$, $\left\|\Id u\right\|_{q,g}<\infty$, if and only if $(X,g)$ is not $\IL^q$-complete. This follows from Theorem 1 and Theorem 2 in \cite{troyan}.
%\end{Remark}

%From here on we will restrict ourselves to $q<\infty$, as $q=\infty$ simply corresponds to geodesic completeness which is precisely the situation we are \emph{not} interested in. 
We first recall the following result. A much thorough discussion can be found in  \cite{troyavo}.

\begin{Proposition}
\label{Stab}
Let $M$ be a smooth manifold. Assume that $g_1$ and $g_2$ are smooth Riemannian metrics on $M$ such that $(M,g_1)$ is $q$-parabolic for some $q<\infty$. Let $A$ be the strictly positive smooth vector bundle endomorphism given by 
$$
A:\IT M\longrightarrow  \IT M, \>g_1(AV_1 ,V_2):=\>g_2(V_1 ,V_2 ),\>\>\text{ $V_1,V_2\in \IT_x M$.}
$$
Let $|(A^{-1})^t|_{g^*_1}$ be the pointwise operator norm of $(A^{-1})^t\in\mathrm{End}(\IT^* M;g^*_1)$, and assume  
$$
\det(A)^{\frac{1}{2}}\cdot|(A^{-1})^t|^{\f{q}{2}}_{g^*_1} \in\IL^{\infty}(M) .
$$
Then $(M,g_2)$ is $q$-parabolic as well. 
\end{Proposition}

\begin{proof} Let $\{\psi_n\}\subset \mathsf{Lip}_{\c}(M)$ be a sequence of functions that makes $g_1$ $q$-parabolic. In particular 
$$
\lim_{n\rightarrow \infty}\int |\Id \psi_n|^{q}_{g_1^*}\Id\mu_{g_1}=0.
$$ 
Then
\begin{align*}
 &\int |\Id \psi_n|^{q}_{g_2^*}\Id\mu_{g_2}=\int g_1^*\Big((A^{-1})^t\Id \psi_n, \Id \psi_n\Big)^{\f{q}{2}} \det(A)^{\frac{1}{2}}\Id\mu_{g_1}\leq  \int |\Id \psi_n|^{q}_{g_1^*}\big|(A^{-1})^t\big|^{\f{q}{2}}_{g^*_1}\det(A)^{\frac{1}{2}}\Id\mu_{g_1} ,
\end{align*}
which, by assumption, goes to zero. We can thus conclude that $(M,g_2)$ is $q$-parabolic as well. 
\end{proof}

We immediately get the following corollary:

\begin{Corollary}
\label{Confdom}
 Let $M$ be a smooth manifold. Assume that $g_1$ is a  smooth Riemannian metric on $M$ such that $(M,g_1)$ is $q$-parabolic for some $q\in [1,\infty)$. Let $g_2$ be another smooth Riemannian metric on $M$ such that one of the two conditions below is fulfilled:
\begin{enumerate}
\item[(i)] $g_1$ and $g_2$ are quasi-isometric
\item[(ii)] $\dim(M)\geq q$ and $g_2=f^2g_1$ where $f:M\rightarrow \mathbb{R}$ is a smooth function which satisfies $0<f^2\leq c$ for some constant $c>0$ 
%\item[(iii)] $ q=2$, $X$ is a complex manifold and $g_1$, $g_2$ are Hermitian metrics, such that $g_2\leq c g_1$ for some constant $c>0$
\end{enumerate}
Then $(M,g_2)$ is $q$-parabolic.
\end{Corollary}

%\begin{proof}
%If $g_1$ and $g_2$ are quasi-isometric then it is immediate to check that Proposition \ref{Stab} applies and therefore $(M,g_2)$ is $q$-parabolic. 
%\\
%For the second case, as in the proof of Prop. \ref{Stab} let us label by $g_1^*$ and $g_2^*$ the metrics induced respectively by $g_1$ and $g_2$ on $\IT^*M$. Under the second set of assumptions we have $g_2^*=f^{-2}g_1^*$ and $\det(A)^{\frac{1}{2}}=f^m$ where $m=\dim(M)$. Therefore when $m\geq q$ the hypothesis of Proposition \ref{Stab} are satisfied and so we can conclude that  $(M,g_2)$ is $q$-parabolic.
%\end{proof}

The situation is pretty different in the case of almost complex manifolds and $q=2$. In this context, as we will see in the next result, the first condition of the Cor. \ref{Confdom} can be largely relaxed.

\begin{Proposition}
\label{almostcomplex}
Let $(M,J)$ be a smooth almost complex manifold of dimension $2m$. Let $h$ be a smooth Riemannian metric on $M$ compatible with $J$, that is $h(JU_1,JU_2)=h(U_1,U_2)$ for every vector fields $U_1,U_2$ on $M$. Assume that $(M,h)$ is $2$-parabolic. Let $\rho$ be another smooth Riemannian  metric on $M$, compatible with $J$,  such that $\rho\leq ch$ for some $c>0$. Then $(M,\rho)$ is $2$-parabolic. 
\end{Proposition}

\begin{proof}
Let $A$ be as in Prop. \ref{Stab} such that $\rho(V_1,V_2)=h(AV_1,V_2)$. Then it is immediate to check that  $JA=AJ$. Let $p\in M$ be any point in $M$. Let $\lambda$ be an eigenvalue of $A_p:T_pM\rightarrow T_pM$, let $E_p(\lambda)$ be the corresponding eigenspace and let $J_p:T_pM\rightarrow T_pM$ be the action of $J$ on $T_pM$.   We know that  $J_p$ preserves $E_p(\lambda)$ because $J_pA_p=A_pJ_p$. Hence we can conclude that the dimension of $E_p(\lambda)$ is even. This in turn tells us that there exist at most $m$ distinct, positive real numbers $0<\lambda_1\leq...\leq\lambda_m$ such that the eigenvalues of $A_p$ are $\{\lambda_1,\lambda_1,...,\lambda_m,\lambda_m\}$. In particular we get that $$\det(A_p)=\prod_{n=1}^m \lambda_n^2.$$ By the fact that $\rho\leq ch$ we easily get that $0<\lambda_n\leq c$ and this immediately yields the following bound:
\begin{equation}
\label{upperbound}
\det(A_p)^{\frac{1}{2}}\cdot|(A_p^{-1})^t|_{h^*}\leq c^{m-1}.
\end{equation} 
Since the right hand side of \eqref{upperbound} does not depend on $p$ we proved that $\det(A)^{\frac{1}{2}}\cdot|(A^{-1})^t|_{h^*} \in\IL^{\infty}(M)$ and thus in virtue of Prop. \ref{Stab} the proof is completed.
% Let us label with $\IT_{\mathbb{C}}X$ the complexified tangent bundle of $X$. Let $h_{\mathbb{C}}$ be the extension of $h$ defined by $h(V_1\otimes z_1,V_2\otimes z_2)=z_1\overline{z_2}h(V_1,V_2)$ and let $\rho_{\mathbb{C}}$ be the analogous extension of $\rho$. Let $A_{\mathbb{C}}$ and $J_{\mathbb{C}}$ be the complex linear extensions of $A$ and $J$. Then, for every pair of complex vector fields $W_1,W_2$ we have $\rho_{\mathbb{C}}(W_1,W_2)=h_{\mathbb{C}}(A_{\mathbb{C}}W_1,W_2)$. As usual let us label by $\IT_{\mathbb{C}}^{1,0}X$ and 
%$\IT_{\mathbb{C}}^{0,1}X$ the eigenspace of $J_{\mathbb{C}}$ corresponding to the eigenvalues $i$ and $-i$ respectively. Then we have $\IT_{\mathbb{C}}X=\IT_{\mathbb{C}}^{1,0}X\oplus \IT_{\mathbb{C}}^{0,1}X$. Moreover by the fact that $JA=AJ$ we have that $J_{\mathbb{C}}A_{\mathbb{C}}=A_{\mathbb{C}}J_{\mathbb{C}}$ and this tells us that $A_{\mathbb{C}}|_{\IT_{\mathbb{C}}^{1,0}X}$ induces a bundle isomorphism $A_{\mathbb{C}}|_{\IT_{\mathbb{C}}^{1,0}X}:\IT_{\mathbb{C}}^{1,0}X\rightarrow \IT_{\mathbb{C}}^{1,0}X$ and analogously  $A_{\mathbb{C}}|_{\IT_{\mathbb{C}}^{0,1}X}$ induces a bundle isomorphism $A_{\mathbb{C}}|_{\IT_{\mathbb{C}}^{0,1}X}:\IT_{\mathbb{C}}^{0,1}X\rightarrow \IT_{\mathbb{C}}^{0,1}X$.
\end{proof}

As an immediate consequence of the previous proposition we get the following:
\begin{Corollary}
Let $(M,h)$ be a smooth complex Hermitian manifold. Let $\rho$ be another smooth Hermitian metric on $M$ such that $\rho\leq ch$ for some $c>0$. If $(M,h)$ is $2$-parabolic then $(M,\rho)$ is $2$-parabolic as well.
\end{Corollary}

%\begin{proof}
%According to  the calculations carried out in \cite{GMMI}, p. 146, we know that the inequality $\rho \leq c h$ implies the following inequality $\|\Id\psi\|^2_{\IL^2\Omega^1(X,\rho)}\leq c'\|\Id\psi\|^2_{\IL^2\Omega^1(X,h)}$ for some $c'>0$ and for each smooth function $\psi:X\rightarrow \mathbb{R}$ with compact support, which immediately gives the result.
%\end{proof}

Finally we discus an issue which arises naturally by the previous propositions. Let $(M,g)$ be a smooth Riemannian manifold which is $2$-parabolic. Let $h$ be another smooth Riemannian metric on $M$ such that $h\leq cg$ for some constant $c>0$. The question that arises now is:
\begin{center} 
  Is then $h$ $2$-parabolic as well? 
\end{center}

In case $M$ is complex and the metrics are Hermitian, we have seen that the answer is yes. In general, clearly we can always find a positive function $f:M\rightarrow \mathbb{R}$ such that $f^2g\leq h\leq cg$. By Corollary \ref{Confdom} we know that $f^2g$ is still $2$-parabolic, at least if $\dim(M)\geq 2$. Thus the Riemannian metric $h$ is bounded above and below by two $2$-parabolic Riemannian metrics. Nevertheless, and somewhat surprisingly, it turns out that in general, the answer to the above question is NO. We give a counterexample on a surface: \vspace{2mm}

Let $\overline{M}$ be a smooth compact surface with boundary. Let $Z$ be the boundary and let $M$ be the interior. Let $\phi:U\rightarrow Z\times[0,1)$ be a collar neighborhood of $Z$. Let $g$ be a smooth Riemannian metric on $M$ such that $(\phi^{-1})^*(g|_U)=\Id x^2+x^2g'$ where $g'$ is a smooth Riemannian metric on $Z$. Let $h$ be another smooth Riemannian metric on $M$ such that $(\phi^{-1})^*(g|_U)=x^2(\Id x^2+g')$. Clearly, for some constant $c>0$, we have $h\leq cg$. Moreover, as we have seen in the previous section, $(M,g)$ is $2$-parabolic.  We want to show now that $(M,h)$ is not $2$-parabolic. The proof is carried out by contradiction. Assume that $(M,h)$ is $2$-parabolic and let $\{\psi_n\}_{n\in \mathbb{N}}\subset \mathsf{Lip}_{\c}(M)$ be a sequence which makes $(M,h)$ $2$-parabolic. Consider a smooth Riemannian metric $h'$ on $M$ such that  $(\phi^{-1})^*(g|_U)=\Id x^2+g'$. A straightforward calculation shows that the same sequence $\{\psi_n\}$ satisfies Prop. \ref{improved} with respect to $(M,h')$. This in turn implies immediately that on $(M,h')$ the Sobolev spaces $\mathsf{W}^{1,2}_0(M,h')$ and $\mathsf{W}^{1,2}(M,h')$ coincide, but this is well-known to be false, see for instance $M= B(0,1) $ where $B(0,1)$ is the Euclidean ball centered in $0$ and of radius $1$.\\We point out moreover that the conclusion that $(M,h')$ is not $2$-parabolic follows also using the criterion provided by Cor. 5.2 in \cite{Troya}. Indeed let $N$ be the open surfaces obtained by gluing $[0,\infty)\times Z$ to the boundary of $\overline{M}$ and let $\rho$ be a Riemannian metric on $N$ that over $[0,\infty)\times N$ is given by $\mathrm{e}^{-2r}\Id r^2+g'$. Since $(U,h'|_U)$ and $([0,\infty)\times Z,\rho|_{[0,\infty)\times Z})$ are isometric it is clear that  $(M,h')$ is $2$-parabolic  if and only if $(N,\rho)$ is $2$-parabolic. Moreover, as $2$-parabolicity on surfaces is stable under a conformal change, $(N,\rho)$ is $2$-parabolic if and only if $(N,\rho')$ is $2$-parabolic where $\rho'$ is given by $\beta\rho$ and $\beta:N\rightarrow \mathbb{R}$ is a positive function that over $[0,\infty)\times Z$ coincides with $\mathrm{e}^{2r}$. This means that over $[0,\infty)\times Z$ $\rho'$ takes the form $\Id r^2+\mathrm{e}^{2r}g'$. Now, as we have $\int_{0}^{\infty}\mathrm{e}^{-r}\Id r<\infty$, we can conclude by Cor. 5.2 in \cite{Troya} that $(N,\rho')$ is not $2$-parabolic.

\vspace{4mm}

%Now we proceed discussing some applications to geodesically incomplete Riemannian manifolds.

\subsection{Open subsets of closed  manifolds}

Consider a smooth compact Riemannian manifold $(M,g)$ of dimension $m$. Let $\Sigma\subset M$ be a subset made of a finite union of closed smooth submanifolds, $\Sigma=\cup_{i=1}^n S_i$ such that, for some $z\geq 2$, each submanifold $S_i$ has codimension greater or equal than $z$, that is $\mathrm{cod}(S_i)\geq z\geq 2$. Let $A$ be defined as $M\setminus  \Sigma$ and consider the restriction of $g$ over $A$, $g|_{A}$.

\begin{Proposition}
\label{Open}
In the above situation, $(A,g|_{A})$ is $q$-parabolic for any $q\in[1,z]$. 
\end{Proposition}

\begin{proof}
We start the proof by showing  that $(A,g|_{A})$ is $z$-parabolic. Define $A_i:= M\setminus  S_i$ and let $s_i$ be the dimension of $S_i$. As a first step we want to prove that $(A_i,g|_{A_i})$ is $z$-parabolic. This follows by applying Th. \ref{straticom} and in particular \eqref{condition1}. Indeed let $p\in S_i$ an arbitrary point. Then we can find an open neighborhood $U$ of $p$ and a diffeomorphism $\Phi:U\rightarrow (0,1)^{m}$ such that $\Phi(U\cap S_i)=\{0\}\times \mathbb{R}^{s_i}$ and  such that $(\Phi^{-1})^*(g|_{U})\sim g_e|_{(0,1)^m}$, that is $(\Phi^{-1})^*(g|_{U})$ is quasi-isometric to the restriction on $(0,1)^m$ of the standard euclidean metric $g_e:=\Id x_{1}^2+...+\Id x_m^2$ of $\mathbb{R}^m$. Now, writing $(0,1)^m$ as $(0,1)^{m-s_i}\times (0,1)^{s_i}$ and using cylindrical coordinates, we can write the metric $g_e$ restricted on $(0,1)^m\setminus \Phi(U\cap S_i)$  as $$\sum_{i=1}^{s_i}\Id x_i^2+\Id r^2+r^2h$$ where $r$ is the usual distance function on $\mathbb{R}^{m-s_i}$ and $h$ is a Riemannian metric on $\mathbb{S}^{m-s_i-1}$. Therefore we can consider the smooth incomplete Riemannian manifold $(A_i,g|_{A_i})$ as the regular part of a compact smoothly Thom-Mather stratified pseudomanifold of depth one given by $M\supset S_i$ endowed with an edge metric. Since in this case we have $\hat{c}=1$ and $\mathrm{cod}(S_i)\geq z$, by applying \eqref{condition1}, we get  $(m-s_i-1)\geq z-1$ and thus we can conclude that $(A_i,g|_{A_i})$ is $z$-parabolic. \\
Now, for each $i=1,...,n$, let   $(\psi_{j,A_{i}})_{j\in\IN}\subset \mathsf{Lip}_{\c}(A_i)$ be a sequence which satisfies the assumptions of Prop. \ref{improved}.  %(we remark that the estimates on $\|\Id\psi_{j,A_i}\|_{\IL^2\Omega^1(A_i,g|_{A_i})}$ are based on an estimate of the volume of a tubular neighborhood of $S_i$ and that the lower bound on the codimension of $S_i$ plays a fundamental role precisely at this point).\\
We define $$0\leq \psi_j:=\prod_{i=1}^n \psi_{j,A_{i}}\leq 1$$ and claim that this sequence makes $(A,g|_{A})$ $z$-parabolic. To see this, note first that for each $j\in \mathbb{N}$,  $\psi_j$ is defined as a product of a finite number of compactly supported Lipschitz functions and therefore is in turn a compactly supported Lipschitz function, and thus $\Id \psi_j$ is well-defined. Clearly the support of $\psi_j$ is contained in $A$ and  $\psi_j\to1$ pointwise. In order to complete the proof we have to show that 
\begin{equation}
\label{limit}
\lim_{j\rightarrow \infty}\int_A  |\Id \psi_j|^z_{g^*|_{A}}\Id\mu_{g|_{A}}=0.
\end{equation}
To this end, note that $\Id \psi_j=\sum_{i=1}^n\phi_i\Id\psi_{j,A_i}$ where $\phi_i$ is given by the product 
$$
\phi_i=\psi_{j,A_1}...\psi_{j,A_{i-1}}\psi_{j,A_{i+1}}...\psi_{j,A_n}.
$$
By the fact that $0\leq \phi_i\leq 1$, in order  to establish \eqref{limit}, we have the following estimate for some $C>0$,
$$
\int_A  |\Id \psi_j|^z_{g^*|_{A}}\Id\mu_{g|_{A}}\leq C\sum^n_{i=1}\int_{A_i}  |\Id \psi_{j,A_i}|^z_{g^*|_{A_i}}\Id\mu_{g|_{A_i}},
$$
which tends to zero as $j\to\infty$ by what we have said above. Hence we can conclude that $(A,g|_{A})$ is $z$-parabolic. Finally, by the fact that $\mu_g(A)<\infty$, we have a continuous inclusion 
$$
\IL^{q_2}\Omega^1(A,g|_{A})\longhookrightarrow \IL^{q_1}\Omega^1(A,g|_{A})\>\text{ for each $1\leq q_1 \leq q _2\leq \infty$},
$$
which implies the desired $q$-parabolicity for $q\in [1,z]$.
\end{proof}

\begin{Remark}
In the previous proposition the case $1<q<z$ is a particular case of \cite{Troya} Cor. 4.1. Moreover a different proof that $M\setminus S_i$ is $2$-parabolic can be found in \cite{ChFe}. As we will see in the next results, the case $\mathrm{cod}(S_i)=2$ provides important applications to the $2$-parabolicity in the setting of complex geometry.
\end{Remark}

\begin{Proposition}
\label{almostcomplex2}
Let  $(M,J)$ be a compact   almost complex manifold. Let $\Sigma\subset M$ be a closed subset such that $\Sigma=\cup_{i=1}^nS_i$ where each $S_i$ is a closed submanifold of $M$ satisfying  $\mathrm{cod}(S_i)\geq 2$. Let $A:=M\setminus \Sigma$ and let  $g$ be a smooth symmetric non-negative section of $T^*M\otimes T^*M\to M$ such that $g$ is compatible with $J$ and $g|_{A}$ is strictly positive (in other words, $g|_{A}$ is a Riemannian metric). Then $(A,g|_{A})$ is $q$-parabolic for $q\in [1,2]$.
\end{Proposition}

\begin{proof}
This follows from Prop. \ref{almostcomplex} and Prop. \ref{Open}.
\end{proof}

\subsection{Hermitian complex spaces}

This section contains applications of Prop. \ref{almostcomplex} and Prop. \ref{Open} to the $q$-parabolicity of complex Hermitian spaces. Complex spaces are a classical topic of complex geometry and we refer to \cite{FiG} and to \cite{GrRe} for a deep development of this subject. Here we recall only what is strictly necessary for our aims. A reduced and paracompact complex space $X$ is said  \emph{Hermitian} if the regular part  $\reg(X):=X\setminus \sing(X)$ carries a Hermitian metric $h$ such that for every point $x\in X$ there exists an open neighborhood $U\ni p$ in $X$, a proper holomorphic embedding of $U$ into a polydisc $\phi: U \rightarrow \mathbb{D}^N\subset \mathbb{C}^N$ and a Hermitian metric $\beta$ on $\mathbb{D}^N$ such that $(\phi|_{\reg(U)})^*\beta=h$. A natural example is provided by any subvariety $V$ of a complex Hermitian manifold $(M,g)$ with the metric given by the restriction of $g$ on $\reg(V)$. 
According to the celebrated work of Hironaka the singularities of $X$ can be resolved. More precisely there exists a compact complex manifold $M$, a divisor $E$ with only normal crossings and a surjective holomorphic map $\pi:M\to X$ such that $\pi^{-1}(\reg(X))=M\setminus D$
$$
\pi|_{M\setminus E}:M\setminus E\longrightarrow \reg(X)
$$
is a biholomorphism. We invite the interested reader to consult \cite{Hiro} and \cite{BiMi}. Here we simply recall that a divisor with only normal crossings is a divisor of the form $D=\sum_iV_i$ where $V_i$ are distinct irreducible smooth hypersurfaces and $D$ is defined in a neighborhood of any point by an equation in local analytic coordinates of the type $z_1\cdot...\cdot z_k=0$.

We are finally in the position to state the next result.

\begin{Theorem}
\label{paraHermitian}
Let $(X,h)$ be an irreducible, compact Hermitian complex space. Let $g$ be a smooth Hermitian metric on $\reg(X)$ such that $g\leq ch$ for some $c>0$. Then $(\reg(X),g)$ is $q$-parabolic for each $q\in [1,2]$.
\end{Theorem}

\begin{proof}
We start proving that  $(\reg(X),h)$ is $2$-parabolic. Let $M,E$ and $\pi$ be as described above. We point out that $M\setminus E$ satisfies the assumptions of Prop. \ref{Open} because $E$ is a divisor with only normal crossings and hence is a finite union of nonsingular  compact complex hypersurfaces  of $M$. Thus each component of $E$ has codimension $2$.
Consider now $\pi^*h$. This is a non negative Hermitian product on $M$ such that $\pi^*h>0$ on $M\setminus E$, that is $\pi^*h$ is a Hermitian metric on $M\setminus E$. 
%Therefore, by the fact that $M$ is compact, we have $\pi^*h\leq c\rho$ for some $c>0$. 
Using Prop. \ref{almostcomplex2} we can thus conclude that $(M\setminus E, \pi^*h)$ is $2$-parabolic. By the fact that
$$
\pi|_{M\setminus E}:(M\setminus E,\pi^*h)\longrightarrow (\reg(X),h)
$$
is an isometry we get  that $(\reg(X),h)$ is $2$-parabolic. Using again Prop. \ref{almostcomplex} we have now that $(\reg(X),g)$ is $2$-parabolic and finally, by the fact that $\mu_g(\reg(X))<\infty$, we can argue as in the proof of Prop. \ref{Open} in order to conclude that $(\reg(X),g)$ is $q$-parabolic for each $q\in [1,2]$.
\end{proof}

\begin{Remark}
In the setting of Theorem \ref{paraHermitian}. A different proof of the fact that $(\reg(X),h)$ is $2$-parabolic has been given in \cite{JR} and follows also by Lemma 1.2 of \cite{NSy}.
\end{Remark}

%Consider an irreducible complex projective variety $V\subset \mathbb{C}\mathbb{P}^m$. This means that $V$ is the zero set of a family of homogeneous polynomials such that it is not possible to decompose $V$ as $V=V_1\cup V_2$ with $V_1\subset V$, $V_2\subset V$, $V\neq V_1$, $V\neq V_2$ and such that $V_1$ and $V_2$ are the zero set of other two families of homogeneous polynomials. Equivalently we can say  that $V$ is a Zariski closed subset of $\mathbb{C}\mathbb{P}^m$ and it is not possible to decompose $V$  as $V=V_1\cup V_2$ with $V_1\subset V$, $V_2\subset V$, $V\neq V_1$, $V\neq V_2$ where $V_1$ and $V_2$ are other two Zariski closed subsets of $\mathbb{C}\mathbb{P}^m$. Our reference for this material is \cite{GHa}. Given an irreducible complex projective variety $V\subset \mathbb{C}\mathbb{P}^m$ we will label by  $\sing(V)$ the singular locus of $V$ and by $\reg(V):= V\setminus  \sing(V)$ the regular part of $V$. The regular part of $V$, $\reg(V)$,  becomes a K\"ahler manifold when we endow it with the K\"ahler metric induced by the Fubini-Study metric of $\mathbb{C}\mathbb{P}^m$. In particular we get an {\em incomplete K\"ahler manifold} when $\sing(V)\neq \emptyset$.

%Simply specify what $X$ and $g$ is. Remark that $(X,g)$ is NOT geodesically complete.
%
%Then something like:
We have the following corollary:

\begin{Corollary}\label{Fubini}
 Let $(M,g)$ be a smooth complex manifold and let $V$ be a compact subvariety of $M$. Let  $g_{V}$ be the Hermitian metric on $\reg(V)$ induced by the restriction of $g$ and let $h$ be any smooth Hermitian metric on $\reg(V)$ such that $h\leq cg_V$ for some $c>0$. Then $(\reg(V),h)$ is $q$-parabolic for any $q\in [1,2]$. 
\end{Corollary}

\begin{proof}
This follows as an immediate consequence of Theorem \ref{paraHermitian}.
%In \cite{LT} or in \cite{KIY} the authors prove  that $(\reg(V),g)$ is $\IL^2$-complete. By the fact that $\mu_g(\reg(V))<\infty$ we have a continuous inclusion 
%$$
%\IL^{q_2}\Omega^1(\reg(V),g)\longhookrightarrow \IL^{q_1}\Omega^1(\reg(V),g)\>\text{ for each $1\leq q_1 \leq q _2\leq \infty$},
%$$
%which proves the claim.
\end{proof}

As a particular case of the previous corollary we have:
\begin{Corollary}
\label{projproj}
Let $V\subset \mathbb{C}\mathbb{P}^n$ be a complex projective variety. Let $g$ be the K\"ahler metric on $\reg(V)$ induced by the Fubini-Study metric of $\mathbb{C}\mathbb{P}^n$ and let $h$ be any smooth Hermitian metric on $\reg(V)$ such that $h\leq cg$ for some $c>0$. Then $(\reg(V),h)$ is $q$-parabolic for $q\in [1,2]$.
\end{Corollary}

\begin{Remark}
In the setting of Cor. \ref{projproj}. The parabolicity of $(\reg(V),g)$ has already been proved in \cite{LT} and in \cite{KIY}. Moreover the stochastic completeness of $(\reg(V),g)$ (which follows from Cor. \ref{projproj}) has already been proved by Li and Tian in \cite{LT} by completely different methods (in fact, by a direct calculation).
\end{Remark}

%Applying Proposition \ref{Stab} and Corollary \ref{Confdom} we have the following generalization:
%
%\begin{Proposition}
%Let $V$ be as above. Let $\tilde{h}$ be any smooth Riemannian metric on $\mathbb{C}\mathbb{P}^m$ and let $h$ be the smooth metric on $\reg(V)$ induced by $\tilde{h}$. Then $(\reg(V),h)$ is $\IL^q$-complete for any $q\in [1,2]$.
%\end{Proposition}
%
%\begin{proof}
%By the fact that $\mathbb{C}\mathbb{P}^m$ is compact we have that $\tilde{h}$ is quasi isometric to the Fubini Study metric. Therefore $h$ is quasi isometric to the metric induced on $\reg(V)$ by the Fubini Study metric. Applying Corollary \ref{Confdom} and Prop. \ref{Fubini} we get that $(\reg(V),h)$ is $\IL^q$-complete for $q\in [1,2]$.
%\end{proof}

%\subsection{Real affine algebraic varieties }
We consider now an irreducible affine real algebraic  variety $V\subset \mathbb{R}^m$. For this topic we refer to \cite{BCR}.
% Analogously to the previous example this means that $V$ is the zero set of a family of  polynomials belongings to $\mathbb{R}[x_1,...,x_m]$ such that it is not possible to decompose $V$ as $V=V_1\cup V_2$ with $V_1\subset V$, $V_2\subset V$, $V\neq V_1$, $V\neq V_2$ and such that $V_1$ and $V_2$ are the zero set of other two families of  polynomials. Equivalently we can say  that $V$ is a Zariski closed subset of $\mathbb{R}^m$ and it is not possible to decompose $V$  as $V=V_1\cup V_2$ with $V_1\subset V$, $V_2\subset V$, $V\neq V_1$, $V\neq V_2$ where $V_1$ and $V_2$ are other two Zariski closed subsets of $\mathbb{R}^m$. Given an irreducible affine real algebraic  variety $V\subset \mathbb{R}^m$ we will label by  $\sing(V)$ the singular locus of $V$ and by $\reg(V):= V\setminus  \sing(V)$ the regular part of $V$. For this topic we refer to \cite{BCR}.
We have the following proposition:
\begin{Proposition}
\label{realvariety}
Let $V\subset \mathbb{R}^m$ be a compact and irreducible real affine algebraic variety. Assume that $\dim(\reg(V))-\dim(\sing(V))\geq 2$. Let $U$ be a relatively compact open neighborhood of $V$ in $\mathbb{R}^m$ and let $g$ be a smooth Riemannian metric on $\mathbb{R}^m$ whose restriction on $U$ is quasi isometric to $g_e$, the standard Euclidean metric on $\mathbb{R}^m$. Finally let $i^*_Vg$ be the metric that $g$ induces on $\reg(V)$ through the inclusion $i:\reg(V)\hookrightarrow \mathbb{R}^m$. Then, for each $q\in[1,2]$, $(\reg(V), i^*_{V}g)$ is $q$-parabolic. 
\end{Proposition}

\begin{proof}
That $(\reg(V), i^*_{V}g_e)$ is $2$-parabolic has been proved by Li and Tian in \cite{LT}. Now, by the fact that $\reg(V)$ has finite volume with respect to $i^*_Vg_e$, we get that $(\reg(V), i^*_{V}g_e)$ is $q$-parabolic for each $q\in [1,2]$. Finally applying Corollary \ref{Confdom} we get that  $(\reg(V), i^*_{V}g)$ is $q$-parabolic, for each $q\in[1,2]$, where $g$ is any Riemannian metric on $\mathbb{R}^m$ quasi isometric to $g_e$ over a relatively compact open neighborhood $U$ of $V$.
\end{proof}

%
%We recall below their definition for the convenience of the reader.  Let $r_i$ be the distance function to $A_i$. Let $\epsilon_n:=\frac{1}{n^2}$ and let $\epsilon'_n:=\mathrm{e}^{-\frac{1}{\epsilon_n^2}}=\mathrm{e}^{-n^4}$. Then we define $\psi_{j,A_i}$ as 
%\begin{equation}
%\label{cut}
%\psi_{j,A_i}:= \left\{
%\begin{array}{ll}
%1 & r_i\geq \epsilon_n\\
%(\frac{r_i}{\epsilon_n})^{\epsilon_n} & 2\epsilon'_n\leq r_i\leq \epsilon_n\\
%(\frac{2\epsilon'_n}{\epsilon_n})^{\epsilon_n}(\frac{r_i}{\epsilon'_n}-1) &  \epsilon'_n\leq r_i\leq 2\epsilon'_n\\
%0 & 0\leq r_i\leq \epsilon'_n
%\end{array}
%\right.
%\end{equation}

\end{document}